\documentclass[10pt, a4paper, twoside]{amsart}

\usepackage{amsthm}
\usepackage{amsmath}
\usepackage{amssymb}
\usepackage{amscd}

\usepackage[all]{xy}
\usepackage{tikz}
\usepackage{tikz-cd}
\usepackage{comment}

\usepackage[backref=page]{hyperref}
\usepackage{cleveref}
\usepackage{adjustbox}
\usepackage{mathtools}

\hypersetup{colorlinks=true,linkcolor=blue,anchorcolor=blue,citecolor=blue}

\newtheorem{thm}{\bf Theorem}[section]
\newtheorem{prop}[thm]{\bf Proposition}
\newtheorem{lem}[thm]{\bf Lemma}

\newtheorem{q}[thm]{\bf Question}

\newtheorem*{thm*}{\bf Theorem}
\newtheorem*{cor*}{\bf Corollary}

\theoremstyle{definition}

\newtheorem{rem}[thm]{\it Remark}

\newtheorem*{df*}{\bf Definition}
\newtheorem*{dfs*}{\bf Definitions}
\newtheorem*{ack*}{\bf Acknowledgements}

\numberwithin{equation}{section}

\def\A{\mathbb{A}}
\def\P{\mathbb{P}}
\def\C{\mathbb{C}}
\def\F{\mathbb{F}}
\def\Q{\mathbb{Q}}

\def\Z{\mathbb{Z}}
\def\R{\mathbb{R}}
\def\G{\mathbb{G}}

\DeclareMathOperator{\alg}{alg}

\DeclareMathOperator{\tors}{tors}
\DeclareMathOperator{\Ker}{Ker}
\DeclareMathOperator{\Image}{Im}
\DeclareMathOperator{\Coker}{Coker}
\DeclareMathOperator{\Gal}{Gal}
\DeclareMathOperator{\nr}{nr}
\DeclareMathOperator{\cl}{cl}
\DeclareMathOperator{\tf}{tf}
\DeclareMathOperator{\cd}{cd}

\DeclareMathOperator{\Gm}{\mathbb{G}_{\operatorname{m}}}

\newcommand{\ang}[1]{\left \langle{#1}\right \rangle}
\newcommand{\ov}{\overline}
\newcommand{\on}{\operatorname}
\newcommand{\mc}{\mathcal}

\title[\tiny Non-injectivity of the cycle class map]{Non-injectivity of the cycle class map in continuous $\ell$-adic cohomology}
\author{Federico Scavia}
\address{Department of Mathematics\\
	University of California\\
	Los Angeles, CA 90095-1555 \\ USA}
\email{scavia@math.ucla.edu}

\author{Fumiaki Suzuki}
\email{suzuki@math.ucla.edu}

\date{December 30th, 2022}
\subjclass[2020]{Primary: 14C25, Secondary: 14F20, 14F43, 11G35}

\begin{document}
	
	\maketitle
	
	\begin{abstract}
 Jannsen asked whether the rational cycle class map in continuous $\ell$-adic cohomology is injective, in every codimension for all smooth projective varieties over a field of finite type over the prime field. As recently pointed out by Schreieder, the integral version of Jannsen's question is also of interest.
	We exhibit several examples showing that the answer to the integral version is negative in general. Our examples also have consequences for the coniveau filtration on Chow groups and the transcendental Abel-Jacobi map constructed by Schreieder.
%	    Schreieder asked whether the cycle class map in integral continuous  $\ell$-adic cohomology is injective in every codimension for all smooth projective varieties over a field of finite type over the prime field. We exhibit several examples showing that the answer is negative in general.
	    %We give examples of smooth projective fourfolds $X$ over a number field for which cycle class map in continuous $\ell$-adic cohomology is not injective on torsion. This answers a question of Schreieder.
	\end{abstract}
	
	\section{Introduction}
		Let $k$ be a field, 
		%$\ov{k}$ and 
		$k_s\subset \ov{k}$ be a separable and an algebraic closure of $k$, 
		respectively, $\ell$ be a prime number invertible in $k$, and $X$ be a smooth projective $k$-variety. For all integers $i$ and $j$, we denote by $CH^i(X)$ the Chow group of codimension $i$ cycles modulo rational equivalence, and by  $H^{i}(X,\Q_{\ell}(j))$ the continuous $\ell$-adic cohomology defined by Jannsen \cite{jannsen1988continuous} (or equivalently, the pro-\'etale cohomology defined by Bhatt and Scholze \cite{bhatt2015proetale}). Motivated by 
		%Beilinson's 
		the Bloch--Beilison
		conjecture on the existence of a certain functorial filtration on $CH^i(X)\otimes_{\Z}\Q$ and its relation to the conjectural theory of mixed motives, Jannsen \cite[Question 2.8]{jannsen1994motivic} asked the following question.
	
	\begin{q}[Jannsen]\label{jannsen}
	Suppose that $k$ is of finite type over its prime field. Is the $\ell$-adic cycle class map
	\[\cl\colon CH^i(X)\otimes_{\Z}\Q_{\ell}\to H^{2i}(X,\Q_{\ell}(i))\]
	%always 
	injective?
	\end{q}	
	
	A positive answer to \Cref{jannsen} would imply 
	%Beilinson's 
	the Bloch--Beilinson
	conjecture \cite[Conjecture 2.1]{jannsen1994motivic} over $k$. More precisely, consider the Hochschild-Serre spectral sequence in continuous $\ell$-adic cohomology \cite[Corollary (3.4)]{jannsen1988continuous}:
	\begin{equation*}\label{hochschild-serre-intro}E_2^{p,q}=H^p(k,H^q(X_{k_s},\Q_{\ell}(i)))\Rightarrow H^{p+q}(X,\Q_{\ell}(i)).\end{equation*}
	The spectral sequence degenerates at the $E_2$ page, and so gives a filtration
		\[\{0\}=F^{i+1}\subset F^i\subset\cdots \subset F^1\subset F^0=H^{2i}(X,\Q_\ell(i)).\]
	where %$F^1$ is the kernel of the pull-back map $H^{2i}(X,\Q_\ell(i))\to H^{2i}(X_{k_s},\Q_\ell(i))$ 
	$F^p/F^{p+1}\simeq H^p(k,H^{2i-p}(X_{k_s},\Q_{\ell}(i)))$ for all $p\geq 0$. 
	%When $i=0,1$, $F^i/F^{i+1}$ is even a submodule of   $H^i(k,H^{4-i}(X_{k(t)_s},\Z_2(2)))$. 
	%If the filtration conjectured by Beilinson on $CH^i(X)$ existed, then it would map to $F^{\cdot}$ under the cycle class map; see \cite[p. 263]{jannsen1994motivic}. Conversely, 
	If \Cref{jannsen} had a positive answer, then the inverse image of $F^{\cdot}$ would be a filtration on $CH^i(X)$ with all properties predicted by 
	%Beilinson, 
	Bloch and Beilinson,
	proving 
	%Beilinson's conjecture; 
	the Bloch--Beilinson conjecture;
	see \cite[Lemma 2.7]{jannsen1994motivic}. 
	%Moreover,
	
	Of course, there is no reason to expect \Cref{jannsen} to have an affirmative answer over an arbitrary field. For example, if $k$ is algebraically closed, the kernel of the cycle class map is the group of homologically trivial cycles modulo rational equivalence, and it is often non-trivial: in particular, the cycle class map factors through algebraic equivalence. However, the situation when $k$ is of finite type over its prime field is very different. Indeed, in this case, Jannsen observed that, as a consequence of the Mordell-Weil theorem, the {\em integral} codimension $1$ cycle class map $\on{Pic}(X)\otimes_{\Z}\Z_{\ell}\to H^2(X,\Z_\ell(1))$ is injective; see \cite[Remark 6.15 (a)]{jannsen1988continuous}. This naturally leads to the following variant of Jannsen's question.
	
	\begin{q}\label{schreieder}
	Suppose that $k$ is of finite type over its prime field. Is the $\ell$-adic cycle class map
	\[\cl\colon CH^i(X)\otimes_{\Z}\Z_{\ell}\to H^{2i}(X,\Z_{\ell}(i))\]
	%always 
	injective?
	\end{q}
	
	As noted by Jannsen, \Cref{schreieder} has an affirmative answer for $i=1$. \Cref{schreieder} is also implicit in work of Saito \cite{saito1992cycle}, who obtained some positive results for $i=2$. 
	Colliot-Th\'el\`ene--Sansuc--Soul\'e \cite{colliot1983torsion} showed that the $\ell$-adic cycle class map is injective on torsion when $i=2$ and the field $k$ is finite. 

  \Cref{schreieder} fits into a constellation of conjectural integral refinements of well-known rational cycle conjectures. 
  These questions go back at least to Totaro \cite{totaro1997torsion}, who suggested that certain Lefschetz-hyperplane properties for Chow groups, originally conjectured rationally by Hartshorne, Nori and Paranjape, should also hold for integral Chow groups. Totaro also showed that Nori connectivity for Chow groups fails on torsion cycles. Later Soul\'e--Voisin \cite{soule2005torsion} showed that Voevodsky's smash nilpotence conjecture fails integrally.
    
    In contrast to these negative results, Schreieder \cite{schreieder2022refined} recently proved that some aspects of the rational conjectures hold in fact integrally. 
    For example, Schreieder proved a torsion analogue of a certain conjecture of Jannsen, asserting that cycles in the kernel of the Abel-Jacobi map have coniveau one; see \cite[Corollary 1.3]{schreieder2022refined}. 
    In his talk at the conference ``G\'eom\'etrie Alg\'ebrique en l'honneur de Claire Voisin,'' held in May 2022 in Paris, he used this result to motivate the general and natural question of to which extent rational cycle conjectures hold integrally, and in particular \Cref{schreieder}.  
    
	\medskip
	
	The purpose of the present work is to show that \Cref{schreieder} has a negative answer in general. We offer examples of very different natures:
	topological (Atiyah--Hirzebruch-style approximations of classifying spaces),
	geometric (products of a Kummer threefold and an elliptic curve),
	and arithmetic (quadrics, norm varieties). As we explain below, our examples exhibit new and interesting behavior of the coniveau filtration on Chow groups and of Schreieder's transcendental Abel-Jacobi map over finitely generated fields.
	% (and in particular on the torsion subgroups of Chow groups)

 	\begin{thm}[\Cref{complete-intersection2}]\label{complete-intersection}
	There exist a finite field (resp. a number field) $k$ and a smooth complete intersection $Y\subset \P^N_k$ of dimension $15$ with a free action of a finite $2$-group $G$ such that, letting $X\coloneqq Y/G$, the cycle class map
	\[
	\cl\colon CH^3(X)[2]\rightarrow H^6(X,\Z_2(3))
	\]
	is not injective.
	\end{thm} 
 	
 	The aforementioned result of Colliot-Th\'el\`ene--Sansuc--Soul\'e shows that $3$ is the least possible codimension in which one can find a torsion counterexample over a finite field. 
	
	The dimension of the examples of \Cref{complete-intersection} is quite large. The following theorem yields examples of smaller dimension over a number field.
  
  \begin{thm}[\Cref{thm-fourfold2}]\label{thm-fourfold}
There exist a number field $k$ and 
%a Kummer threefold $Y$ (resp. an elliptic curve $E$) over $k$ such that, letting $X\coloneqq Y\times E$,
a fourfold product $X=Y\times E$ over $k$, where $Y$ is a Kummer threefold and $E$ is an elliptic curve,
such that
the cycle class map \[\cl:CH^3(X)[2]\rightarrow H^6(X,\Z_2(3))\] is not injective.
\end{thm} 

  The examples of \Cref{thm-fourfold} are the counterexamples of smallest dimension that we could find over number fields. 
  Over a field of transcendence degree $1$ over $\Q$, 
  we provide examples of one dimension lower, in one codimension lower.
  Recall that a global field is said to be totally imaginary if it admits no real places. 
    
    \begin{thm}[\Cref{quadric2}]\label{quadric}
    Let $k$ be a totally imaginary number field
    and $k(t)$ be a purely transcendental extension of $k$ of transcendence degree $1$.
    There exists a smooth quadric hypersurface $X\subset \P^4_{k(t)}$ such that the cycle class map
    \[
    \cl\colon CH^2(X)[2]\rightarrow H^4(X,\Z_2(2))
    \]
    is not injective.
\end{thm}
	We also show that, if $\ell$ is an odd prime invertible in $k$, there exists a norm variety $X$ of dimension $\ell^2-1$ over $k(t)$ such that $\cl\colon CH^2(X)[\ell]\rightarrow H^4(X,\Z_\ell(2))$ is not injective. Thus \Cref{schreieder} has a negative answer for all prime numbers $\ell$.
	
	\medskip

    We now explain the relation of our examples to Schreieder's results on the coniveau filtration on Chow groups. 
    By now, we have a good understanding of the filtration over the complex numbers, especially for codimension $\leq 3$; see after \cite[Corollary 1.2]{schreieder2022refined}.
    %By now we have examples that explain how that filtration behaves for codimension $3$ cycles over $k=\C$; see after \cite[Corollary 1.2]{schreieder2022refined}. 
    Our examples show that this filtration is still interesting when $k$ is of finite type over its prime field.
    %finitely generated. 
    We also relate our examples to the transcendental Abel-Jacobi map on torsion cycles constructed by Schreieder \cite[\S 7.5]{schreieder2022refined}. 

    \begin{rem}
    We denote by $N^{\cdot}CH^i(X)$ the coniveau filtration on Chow groups \cite[\S 1.1]{schreieder2022refined}, 
    and by $H^i_{j,\on{nr}}(X,-)$ Schreieder's refined unramified cohomology \cite[\S 5]{schreieder2022refined}, which for $j=0$ coincides with the ordinary unramified cohomology: $H^i_{0,\on{nr}}=H^i_{\on{nr}}$.
    In the following, we assume that $k$ is of finite type over its prime field.
    
    (a) 
    We have $N^{i-1}CH^i(X)\otimes_\Z{\Z_\ell}=0$ for all smooth projective $k$-varieties $X$ by Jannsen's result; see \cite[Lemma 7.5(2)]{schreieder2022refined}.
    The examples of \Cref{quadric} show that $N^{i-2}CH^{i}(X)\otimes_{\Z} \Z_{\ell}$ can be non-zero for $i=2$. 
    (In this case $N^{0}CH^{2}(X)\otimes_{\Z}\Z_{\ell}$ is exactly the kernel of the cycle class map.)
    One can further analyze the torsion part of the stage of the filtration on the examples using the transcendental Abel-Jacobi map, and \cite[Corollary 9.5, Proposition 7.16]{schreieder2022refined} yields
    \[N^1H^3(X,\Q_2/\Z_2(2))_{\on{div}}/N^1H^3(X,\Q_2(2))\otimes \Q_2/\Z_2 \neq 0.\]
    In other words, there is a cohomology class in $H^3(X,\Q_2/\Z_2(2))$ of coniveau $1$, which lifts to a rational class, but not to a rational class of coniveau $1$.    
    
    (b) By \cite[Theorem 1.8]{schreieder2022refined}, the kernel of the cycle class map is given by 
    \[H^{2i-1}_{i-2,\on{nr}}(X,\Z_\ell(i))/H^{2i-1}(X,\Z_\ell(i)).\]
    \cref{jannsen} asks whether this group is torsion.    
    Our examples of Theorems \ref{complete-intersection}, \ref{thm-fourfold}, and \ref{quadric} show that it can be nonzero for $i=2,3$.
    %$2\leq i\leq \dim X-1$.  
   % By $H^i_{0,\on{nr}}=H^i_{\on{nr}}$ for all $i$, 
    In the case $i=2$, we get an explicit statement on ordinary unramified cohomology:
    the examples of \Cref{quadric} have an unramified class of degree $3$ which does not extend to a class on all of $X$. In fact, using 
    a restriction-corestriction argument,
    %a rational decomposition of the diagonal, 
    one sees that in this case the inclusion 
    \[H^3(k,\Z_2(2))\subset H^3_{\on{nr}}(X,\Z_2(2))\] 
    %is of finite index 
    has cokernel of finite torsion order $>1$, a phenomenon that does not seem to have been observed before. (In contrast,
%    $H^3(k,\Z_\ell(2))$
    $H^3_{\on{nr}}(X,\Z_2(2))$
    is torsion-free, and $H^3(k,\Z_2(2))$ is a direct summand of $H^3_{\on{nr}}(X,\Z_2(2))$ if $X(k)\neq \emptyset$.)    
    
    (c) 
    In our setting, Schreieder's transcendental Abel-Jacobi map is of the form: 
    $\lambda_{\text{tr}}\colon CH^{i}_{0}(X)\{\ell\}\rightarrow H^{2i-1}(X,\Q_{\ell}/\Z_{\ell}(i))/N^{i-1}H^{2i-1}(X,\Q_{\ell}(i))$,
    where $CH^{i}_{0}(X)\{\ell\}$ is the kernel of $\cl\colon CH^{i}(X)\{\ell\}\rightarrow H^{2i}(X,\Z_{\ell}(i))$.    
    For $i=2$, the transcendental Abel-Jacobi map is injective by \cite[Corollary 9.5]{schreieder2022refined}.
    In particular, the torsion cycles in the examples of \cref{quadric} do not lie in the kernel of $\lambda_{\text{tr}}$.
    This also shows that $\lambda_{\text{tr}}$ can be non-zero.
    % (cf. \cref{prop-taj}).
    In contrast, Theorems \ref{complete-intersection} and \ref{thm-fourfold} provide examples where $\lambda_{\text{tr}}$ is not injective for $i=3$; see \cref{rem-taj}.

    \end{rem}
    
%	\medskip
	
	We now comment on the proofs of the main theorems. In view of the discussion around \Cref{jannsen}, it is natural to approach \Cref{schreieder} by considering the filtration $F^{\cdot}$ on $H^{2i}(X,\Z_\ell(i))$ induced by the Hochschild-Serre spectral sequence
	\begin{equation*}\label{hochschild-serre-intro2}E_2^{p,q}=H^p(k,H^q(X_{k_s},\Z_{\ell}(i)))\Rightarrow H^{p+q}(X,\Z_{\ell}(i)).\end{equation*}
	%$F^{\cdot}$. 
	We start with a non-zero torsion cycle $\alpha\in CH^i(X)$ (producing such examples is generally quite difficult) and try to show that $\cl(\alpha)\in F^p$ for all $p\geq 0$. To show that $\cl(\alpha)\in F^1$, we only need to show that $\cl(\alpha)$ is geometrically trivial, but the subsequent steps of the filtration are more difficult because the groups appearing in the spectral sequence are typically huge and the image of $\cl(\alpha)\in F^p/F^{p+1}$ often seems hard to compute; see \cite{jannsen2000letter} for the case $p=2$. In the examples used to prove  Theorems \ref{complete-intersection} and \ref{quadric}, we get around this by showing that all $F^p/F^{p+1}$ are torsion-free, which forces $\cl(\alpha)=0$.
	
	\Cref{thm-fourfold} lies deeper. A key result (\cref{thm-blochjannsen}), relating injectivity of the $\ell$-adic cycle class map to that of Bloch's map,
	%shows that if the $\ell$-adic cycle class map is injective on torsion over all number fields, then the Deligne cycle class map is injective on torsion over the complex numbers. \cref{thm-fourfold} is then reduced 
	reduces Theorem \ref{thm-fourfold}
	to finding fourfold examples defined over a number field where the Deligne cycle class map is not injective on torsion in codimension $3$.
	We then achieve this in two steps: a result of Bloch--Esnault yields examples defined over a number field with non-vanishing fourth unramified cohomology group $H^4_{\nr}(X_{\C},\Q_\ell/\Z_\ell(3))$, where, with extra care, one can find such examples with small Chow group of zero-cycles; then using the Bloch--Kato conjecture and a result of Voisin and Ma relating $H^4_{\nr}(X_{\C},\Q_\ell/
	\Z_\ell(3))$ to the kernel of the Deligne cycle class map on torsion in codimension $3$, one deduces the desired non-injectivity. The construction 
	%of actual examples 
	is inspired by the work of Diaz.
	
	\medskip
	
    Our work leads us to the  following questions.

	\begin{q}\label{our-question}
	(a) Is there a smooth projective $d$-dimensional variety $X$ over a field of finite type over its prime field such that the $\ell$-adic map $\cl\colon CH^d(X)\otimes_{\Z}\Z_{\ell}\to H^{2d}(X,\Z_{\ell}(d))$ is not injective? \footnote{After the first version of this manuscript was posted on arXiv, Alexandrou and Schreieder annouced a construction of such $d$-folds for all $d\geq 3$; see \cite[Corollary 1.4]{alexandrou2022bloch}. Later, Colliot-Th\'el\`ene and the first author \cite{colliot2022injectivite} found an example for $d=2$.} 
 
	(b) Let $i$ be either $2$ or $3$. Is there a smooth projective threefold over a number field $k$ such that the $\ell$-adic map $\cl\colon CH^i(X)\otimes_{\Z}\Z_{\ell}\to H^{2i}(X,\Z_{\ell}(i))$ is not injective? What happens over $k=\Q$?
%	(c) Can one produce examples of non-injectivity of the cycle class map for regular schemes that are quasi-projective over $\Z[1/\ell]$?
	\end{q}
	
	The paper is organized as follows. In \Cref{section-ci}, we prove \cref{complete-intersection}.
	%construct topological examples (\cref{complete-intersection2}).
	In \Cref{section-JB}, we prove a key result (\cref{thm-blochjannsen}), relating the injectivity of the $\ell$-adic cycle class map to that of Bloch's map, which is useful in Sections \ref{section-fourfold} and \ref{section-furtherex}.
	As first application, we give a second proof of \cref{complete-intersection}.
	In \Cref{section-fourfold}, we prove \cref{thm-fourfold}.
	%construct geometric examples (\cref{thm-fourfold2}).
	In \Cref{section-furtherex}, we construct further examples in codimension $3$ using non-torsion type counterexamples to the integral Hodge and Tate conjectures.
	Finally, in \Cref{section-quadric}, we prove \cref{quadric}.
	%construct arithmetic examples (\cref{quadric2}).
	
	\subsection*{Notation}
	If $k$ is a field, we write $H^i(k,-)$ for continuous Galois cohomology. If $X$ is a smooth projective $k$-variety, we write $H^i(X,-)$ for the continuous \'etale cohomology, as defined by Jannsen \cite{jannsen1988continuous}, $CH^i(X)$ for the Chow group of codimension $i$ cycles modulo rational equivalence, and
	$\cl$ for the cycle class map in continuous $\ell$-adic cohomology;
	when $k$ is algebraically closed, 
	we write $\lambda$ for Bloch's map.
	%and $CH^i(X)_{\alg}\subset CH^i(X)$ for the subgroup of algebraically trivial cycle classes.
	If $k=\C$, we denote by $H^i_{\mc{D}}(X,\Z(j))$ the Deligne cohomology group
	and by $\cl_{\mc
	D}$ the Deligne cycle class map;
	for $A\in\{\Z,\Q/\Z,\Z/2\}$, we denote by $H^i_{\on{nr}}(X,A)$ the $i$-th unramified cohomology group.
	%$A^i(X)$ for the group of codimension $i$ cycles modulo algebraic equivalence.
	%Deligne cohomology,
	
	For an abelian group $A$, an integer $n\geq 1$, and a prime number $\ell$, we denote $A[n]\coloneqq \left\{a\in A\mid na=0\right\}$, by $A\{\ell\}$ the subgroup of $\ell$-primary torsion elements of $A$, by $A_{\tors}$ the subgroup of torsion elements of $A$, and $A_{\tf}\coloneqq  A/A_{\tors}$.

		\section{Proof of Theorem \ref{complete-intersection}}\label{section-ci}
	
	In order to prove \Cref{complete-intersection}, 
	%In order to construct topological examples,
	we will make use of a construction due to Totaro \cite{totaro1997torsion}. Totaro's construction is stated over the complex numbers but works over an arbitrary field of characteristic zero. It has been generalized to fields of characteristic not $2$ by Quick \cite{quick2011torsion}.
	
	Let $k_0$ be a field of characteristic different from $2$.  Let $H$ be the Heisenberg group of order $32$ (see \cite[\S 5]{totaro1997torsion}) and set $G\coloneqq H\times \Z/2$. We have a group homomorphism
	\[\varphi\colon G\xrightarrow{\on{pr}_1}H\hookrightarrow \on{SO}_4,\]
	where the map on the right is the Heisenberg representation of $H$; see \cite[\S 5]{totaro1997torsion}. (Totaro works in characteristic zero, but as observed during the proof of \cite[Theorem 7.2]{quick2011torsion}, the Heisenberg representation is defined over any field of characteristic different from $2$.)
	Let $A:\on{SO}_4\to \on{GL}_3$ be the representation given by the composition
	\[\on{SO}_4\twoheadrightarrow \on{SO}_4/\mu_2\xrightarrow{\sim}\on{SO}_3\times\on{SO}_3\xrightarrow{\on{pr}_1} \on{SO}_3\hookrightarrow\on{GL}_3,\]
	and $B:\on{SO}_4\to\on{GL}_4$ be the natural $4$-dimensional representation of $\on{SO}_4$. Define \[C\coloneqq c_2(A\circ\varphi)-c_2(B\circ\varphi)\in CH^2(BG),\] let $c_1\in CH^1(BG)$ be the pullback along the second projection $\on{pr}_2:G\to \Z/2$ of the first Chern-class of the non-trivial character of $\Z/2$, and set \[\alpha\coloneqq Cc_1\in CH^3(BG).\] We have $2\alpha=0$ because $2c_1=0$.
	
	Finally, let $V$ be a  $G$-representation of finite dimension over $k_0$, $U\subset V$ be a $G$-invariant open subscheme of $V$ such that $G$ acts freely on $U$ and the codimension of $V-U$ in $V$ is at least $4$.
	
	\begin{lem}\label{u-mod-g}
	Let $(k_0)_s$ be a separable closure of $k_0$.
	
	(a) We have $\cl(\alpha_{(k_0)_s})=0$ in $H^6((U/G)_{(k_0)_s},\Z_2(3))$.
	
	(b) There exists a finite field subextension $k_0\subset k\subset (k_0)_s$ such that $\cl(\alpha_k)=0$ in $H^6((U/G)_k,\Z_2(3))$.
	\end{lem}

	\begin{proof}
	Since the codimension of $V-U$ in $V$ is at least $4$,  we have \[CH^3(U/G)=CH^3(BG);\]
    see \cite[Definition 1.2]{totaro1999chow}.
    
    (a) By the invariance of \'etale cohomology under purely inseparable field extensions, it suffices to show that $\cl(\alpha_{\ov{k}_0})=0$ in $H^6((U/G)_{\ov{k}_0},\Z_2(3))$, where $\ov{k}_0$ is an algebraic closure of $k_0$ containing $(k_0)_s$. If $k_0=\C$, the map $U/G\to BG$ corresponding to the principal $G$-bundle $U\to U/G$ induces an isomorphism $H^6((U/G)_{\C},\Z)\xrightarrow{\sim} H^6(BG,\Z)$, and the cycle class of $\alpha$ in $H^6(BG,\Z)$ is zero as stated in \cite[p. 485]{totaro1997torsion}, hence the cycle class of $\alpha$ in $H^6((U/G)_{\C},\Z)$ vanishes. Since Artin's comparison isomorphism is compatible with cycle classes in singular and $\ell$-adic cohomology, this implies that (a) holds for $k_0=\C$. If $k_0$ is an arbitrary field of characteristic zero, then (a) follows from the case $k_0=\C$ and the invariance of $\ell$-adic cohomology under extensions of algebraically closed fields. Finally, if $k_0$ is an arbitrary field of characteristic different from $2$, the arguments of Totaro have been adapted by Quick using \'etale cobordism; see the proof of \cite[Proposition 5.3]{quick2011torsion}. One could also argue more directly via a specialization argument from the characteristic zero case. This completes the proof of (a).
	%We follow the proof of \cite[Theorem 7.2]{totaro1997torsion}, and the variant in positive characteristic \cite[Proposition 5.3(b)]{quick2011torsion}. 
	 %It is proved in  that  the cycle class of $\alpha$ in $H^6((U/G)_{\ov{k}_0},\Z_2(3))$ vanishes. (If $k$ is a subfield of $\C$, we may instead use \cite[Theorem 7.2]{totaro1997torsion}, Artin's comparison theorem and the invariance of \'etale cohomology under extensions of algebraically closed fields.) 
	
	(b) The morphism $U_{(k_0)_s}\to (U/G)_{(k_0)_s}$ is a Galois $G$-cover, hence we have the Hochschild-Serre spectral sequence in $\ell$-adic cohomology
	\begin{equation}\label{hochschild-serre2}E_2^{i,j}=H^i(G,H^j(U_{(k_0)_s},\Z_2(3)))\Rightarrow H^{i+j}((U/G)_{(k_0)_s},\Z_2(3)).\end{equation}
	Here $H^i(G,-)$ denotes group cohomology. Since $U$ is an open subscheme of a vector space whose complement has codimension $\geq 4$, we have $H^0(U_{(k_0)_s},\Z_2(3))=\Z_2(3)$ and $H^j(U_{(k_0)_s},\Z_2(3))=0$ for all $1\leq j\leq 6$. We deduce that the natural map $H^i(G,\Z_2(3))\to H^i((U/G)_{(k_0)_s},\Z_2(3))$ is an isomorphism for all $1\leq i\leq 6$.  Since the group $G$ is finite, the group \[H^i(G,\Z_2(3))\simeq H^i(G,\Z_2)\simeq H^i(G,\Z)\otimes_{\Z}\Z_2\] is finite for all $i\geq 1$, hence %$H^i((U/G)_{(k_0)_s},\Z_2(3))$ is finite for all $1\leq i\leq 6$. From the invariance of \'etale cohomology under purely inseparable field extensions, we deduce that
	\begin{equation}\label{finite-coh}
	    \text{$H^i((U/G)_{(k_0)_s},\Z_2(3))$ is finite for all $1\leq i\leq 6$.}
	\end{equation}
	
	For every finite field subextension $k_0\subset k\subset (k_0)_s$, the Hochschild-Serre spectral sequence in continuous $\ell$-adic cohomology %\cite[Corollary (3.4)]{jannsen1988continuous}
		\begin{equation}\label{hochschild-serre}E_2^{i,j}=H^i(k,H^j((U/G)_{(k_0)_s},\Z_2(3)))\Rightarrow H^{i+j}((U/G)_k,\Z_2(3))\end{equation}
		yields a filtration
		\[\{0\}=F^7\subset F^6\subset\cdots \subset F^1\subset F^0=H^6((U/G)_k,\Z_2(3))\]
		where $F^i/F^{i+1}$ is a subquotient of $H^i(k,H^{6-i}((U/G)_{(k_0)_s},\Z_2(3)))$. When $i=0,1$, $F^i/F^{i+1}$ is even a submodule of   $H^i(k,H^{6-i}((U/G)_{(k_0)_s},\Z_2(3)))$.

	It is a consequence of \cite[I.2.2, Corollary 1]{serre1997galois} that for all $i\geq 1$ and all  finite continuous $\on{Gal}((k_0)_s/k)$-modules $M$, any element of $H^i(k,M)$ is killed by passage to a suitable finite extension of $k$. Thus (\ref{finite-coh}) implies that for all $1\leq i\leq 6$, any element of $H^i(k,H^{6-i}((U/G)_{(k_0)_s},\Z_2(3)))$ vanishes after base change to a suitable finite extension of $k$. By (a), we know that $\cl(\alpha_{(k_0)_s})=0$, that is, $\cl(\alpha)\in F^1$. Using the fact that $F^i/F^{i+1}$ is a subquotient of $H^i(k,H^{6-i}((U/G)_{(k_0)_s},\Z_2(3)))$, we may now construct finite field extensions \[k_0\subset k_1\subset \dots\subset k_6\subset k_7\] such that $\cl(\alpha_{k_i})\in F^i$ for all $i$. In particular, $\cl(\alpha_{k_7})=0$, hence $k=k_7$ satisfies the conclusion of the lemma.
	\end{proof}

 \begin{rem}
    It is important to note that continuous $\ell$-adic cohomology does not commute with inverse limits of schemes, so (b) is not a formal consequence for (a). %as can already be seen in the case of curves.
	\end{rem}
	
	Here is a generalized version of \cref{complete-intersection}.
	
	\begin{thm}\label{complete-intersection2}
	Let $k_0$ be a field of characteristic different from $2$. There exist a finite $2$-group $G$, a smooth complete intersection $Y\subset \P_{k_0}^N$ of dimension $15$ with a free $G$-action and finite extension $k/k_0$ such that, letting $X\coloneqq Y/G$, the cycle class map \[\cl\colon CH^3(X_k)[2]\to H^{6}(X_k,\Z_2(3))\] is not injective.
	\end{thm}
	\begin{proof}
	%[Proof of \Cref{complete-intersection}]
	Let $Y$ be a smooth complete intersection of dimension $15$ over $k_0$ on which $G\coloneqq H\times \Z/2$ acts freely, and set $X\coloneqq Y/G$: see \cite[Proposition 15]{serre58charp}. Letting $G$ act diagonally on $Y\times U$, the projections of $Y\times U$ onto its factors are $G$-equivariant: we write $\pi_1\colon(Y\times U)/G\to X$ and $\pi_2\colon(Y\times U)/G\to U/G$ for the induced morphisms.	We have a commutative diagram 
	\[
	\begin{tikzcd}
	CH^3(X_k) \arrow[d,"\cl"] \arrow[r, "\pi_1^*"] & CH^3(((Y\times U)/G)_k) \arrow[d,"\cl"]  & CH^3((U/G)_k) \arrow[d,"\cl"]  \arrow[l,swap, "\pi_2^*"] \\
	H^6(X_k,\Z_2(3)) \arrow[r, "\pi_1^*"]  &  H^6(((Y\times U)/G)_k,\Z_2(3))   & H^6((U/G)_k,\Z_2(3)). \arrow[l,swap, "\pi_2^*"]
	\end{tikzcd}
	\]
	The projection $Y\times V\to Y$ is a $G$-equivariant vector bundle and the $G$-action on $Y$ is free, therefore by descent and Grothendieck's version of Hilbert's Theorem 90 the induced morphism $(Y\times V)/G\to X$ is also a vector bundle. Since $Y\times U\to Y$ is a $G$-invariant dense open subscheme of the $G$-equivariant vector bundle $Y\times V\to Y$, %hence 
	$\pi_1$ is a dense open subscheme of a vector bundle. Moreover, since $V-U$ has codimension $\geq 4$ in $V$, the codimension of the complement $(Y\times U)/G$ inside $(Y\times V)/G$ is also $\geq 4$, hence by \cite[Theorem 3.23]{jannsen1988continuous} and homotopy invariance the maps $\pi_1^*$ are isomorphisms. We get a well-defined element \[\beta\coloneqq (\pi_1^*)^{-1}(\pi_2^*(\alpha_k))\in CH^3(X_k)[2].\] By \Cref{u-mod-g}(b) we have $\cl(\beta)=0$. In order to complete the proof, it remains to show that $\beta\neq 0$.
	
	Suppose first that $k=\C$. Then Totaro showed in \cite{totaro1997torsion} that the class of $\beta$ in the complex cobordism group $MU^6(X)\otimes_{MU^*(X)}\Z$ is not zero, hence $\beta\neq 0$. If $k$ is a field of characteristic zero, the rigidity of the $2$-torsion subgroup of the Chow group \cite{lecomte1986rigidite} implies $\beta_{\ov{k}}\neq 0$, hence $\beta\neq 0$. If $k$ has positive characteristic (different from $2$), the arguments of Totaro have been adapted by Quick; see the proof of \cite[Proposition 5.3(b)]{quick2011torsion}. We conclude that $\beta\neq 0$, as desired.
	\end{proof}

	\section{\texorpdfstring{$\ell$}{l}-adic cycle class map and Bloch's map}\label{section-JB}
	
	In this section we explain the relation between the cycle class map in continuous $\ell$-adic cohomology and a certain map defined by Bloch.
	The main result of this section (\cref{thm-blochjannsen}) will be used to produce counterexamples to \cref{schreieder} in Sections \ref{section-fourfold} and \ref{section-furtherex}.
	
	Let $k_0$ be a field, $i\geq 0$ be an integer, $\ell$ be a prime number invertible in $k_0$, and $X$ be a smooth projective $k_0$-variety.
	For every finite extension $k/k_0$, we have the cycle class map
	$\cl_k\colon CH^i(X_k)\otimes_{\Z}\Z_{\ell}\rightarrow H^{2i}(X_k,\Z_\ell(i))$
	and the Bockstein homomorphism
	\[\beta_k\colon H^{2i-1}(X_k,\Q_\ell/\Z_\ell(i))\rightarrow H^{2i}(X_k,\Z_\ell(i)).\]
	It will be important for us that 
	$CH^i(X_{\ov{k}_0})=\varinjlim_{k/k_0}CH^i(X_k)$ and
	\[
	H^{2i-1}(X_{\ov{k}_0},\Q_\ell/\Z_\ell(i))=\varinjlim_{k/k_0}H^{2i-1}(X_{k},\Q_\ell/\Z_\ell(i)),
	\]
	where the direct limits are over all 
	finite extensions $k/k_{0}$
	contained in $\ov{k}_0/k_0$.
	%finite subextensions $k_{0}\subset k \subset \ov{k}_{0}$.
	Finally, recall that Bloch \cite{bloch1979torsion} (also see \cite{colliot1993cycles}) defined a map
	\[
	\lambda\colon CH^i(X_{\ov{k}_0})\{\ell\}\rightarrow H^{2i-1}(X_{\ov{k}_0},\Q_\ell/\Z_\ell(i)),
	\]
	which, for $\ov{k}_{0}=\C$, coincides with the Deligne cycle class map on torsion \cite[Proposition 3.7]{bloch1979torsion}.
	Note that $\lambda$ is rigid, that is, it does not change under algebraically closed field extensions, because the rigidity property holds for the torsion part of Chow groups \cite{lecomte1986rigidite} and for \'etale cohomology with torsion coefficients.
	
	\begin{prop}\label{thm-blochjannsen}
	The composition
	\[
	CH^i(X_{\ov{k}_0})\{\ell\}\xrightarrow{\lambda}H^{2i-1}(X_{\ov{k}_0},\Q_\ell/\Z_\ell(i)) \xrightarrow{\varinjlim_{k/k_0}\beta_k}\varinjlim_{k/k_0} H^{2i}(X_k,\Z_\ell(i))
	\]
	coincides with $\varinjlim_{k/k_0} \cl_k$ on torsion.
	If $k_{0}$ is of finite type over its prime field, 
	$\varinjlim_{k/k_0}\beta_k$ induces an isomorphism
	\[
	H^{2i-1}(X_{\ov{k}_0},\Q_\ell/\Z_\ell(i))\xrightarrow{\sim}\left(\varinjlim_{k/k_0} H^{2i}(X_k,\Z_\ell(i))\right)\{\ell\},
	\]
	hence $\varinjlim_{k/k_0}\cl_k$ is injective on torsion if and only if $\lambda$ is injective.
	\end{prop}
	
	\begin{rem}\label{rem-blochjannsen}
%	Suppose that $k_{0}$ is of finite type over its prime field. 
	%\cref{thm-blochjannsen} then shows that the injectivity of $\varinjlim_{k/k_{0}}\cl_{k}$ on torsion is equivalent to the injectivity of $\lambda$.
	
	For $i\in\{1,2,\dim X\}$, $\lambda$ is injective:
	the case of $i=1$ is elementary using the Kummer sequence \cite[Proposition 3.6]{bloch1979torsion}, the case of $i=\dim X$ is due to Rojtman \cite{rojtman1980torsion} (see also \cite[Theorem 4.2]{bloch1979torsion}), 
	and the case of $i=2$ is a consequence of a theorem of Merkurjev--Suslin \cite[\S 18]{merkurjev1982k-cohomology}.
	%Hence 
	In these cases,
	if $k_0$ is of finite type over its prime field,
%	in these cases,
	$\varinjlim_{k/k_{0}}\cl_{k}$ is injective on torsion
	by \cref{thm-blochjannsen}.
	For $i=1$, this is also a direct consequence of the observation of Jannsen \cite[Remark 6.15 (a)]{jannsen1988continuous} that $\cl_k\colon CH^1(X_k)\otimes_{\Z} \Z_\ell\rightarrow H^2(X_k,\Z_\ell(1))$ is injective. 
	Remarkably, the kernel of $\cl_k\colon CH^2(X_k)\{\ell\}\rightarrow H^4(X_k,\Z_\ell(2))$ might be non-zero, as we will see in \Cref{section-quadric}.
	
	For $3\leq i\leq \dim X-1$,	there are several known examples \cite{totaro1997torsion,schoen2000certain,soule2005torsion,rosenschon2010griffiths,totaro2016complex,schreieder2022infinite} where $\lambda$ is not injective;
	among them, 
	\cite{schoen2000certain,rosenschon2010griffiths,totaro2016complex,schreieder2022infinite} even showed that the kernel of $\lambda$ may be infinite.
	Note that minimal fields of definition for \cite{schoen2000certain,soule2005torsion,rosenschon2010griffiths,totaro2016complex,schreieder2022infinite} have positive transcendence degree over $\Q$,
	while Totaro's $15$-dimensional examples in \cite{totaro1997torsion} may be defined over $\Q$ or $\F_p$ with $p\neq 2$. 
	In \Cref{section-fourfold}, we exhibit the first fourfold examples defined over $\Q$
	where $\lambda$ is not injective over  $\ov{\Q}$.
	($4$ is the least possible dimension in which one can find such an example.)
	In \Cref{section-furtherex}, we give further instances of non-injectivity of $\lambda$ in relation to the integral Hodge and Tate conjectures.
	Using \cref{thm-blochjannsen} and the rigidity property of $\lambda$, all of these provide counterexamples to \cref{schreieder} over all sufficiently large finite extensions of minimal fields of definition.
	\end{rem}
	
	\begin{proof}[Proof of \cref{thm-blochjannsen}]
	The second assertion follows 
	%from the first one 
	by 
	%from
	observing that if $k_{0}$ is of finite type over its prime field, %$\varinjlim_{k/k_{0}}\beta_{k}$
	%yields an isomorphism:
	%\[
	%H^{2i-1}(X_{\ov{k}_0},\Q_\ell/\Z_\ell(i))\xrightarrow{\sim}\left(\varinjlim_{k/k_0} H^{2i}(X_k,\Z_\ell(i))\right)\{\ell\},
	%\]
	%which may be deduced from the fact that 
	then
	for every finite 
	%subextension $k_{0}\subset k \subset \ov{k}_{0}$
	extension $k/k_0$ contained in $\ov{k}_0/k_0$
	the map $H^{2i-1}(X_{k},\Q_\ell(i))\rightarrow H^{2i-1}(X_{\ov{k}_0},\Q_\ell/\Z_\ell(i))$
	is zero, because it factors through $H^{2i-1}(X_{\ov{k}_0},\Q_\ell(i))^{\Gal(\ov{k}_0/k)}$ which vanishes by weight reasons. 		
	
	It remains to show the first assertion.
	By construction, $\lambda$ fits into the commutative diagram:
	\[
	\begin{tikzcd}
	H^{i-1}(X_{\ov{k}_0},\mc{H}(\Q_\ell/\Z_\ell(i)))\ar[twoheadrightarrow,"f"]{r}\ar[d, "g"]&CH^i(X_{\ov{k}_0})\{\ell\}\ar[ld,"-\lambda"]\\
	H^{2i-1}(X_{\ov{k}_0},\Q_\ell/\Z_\ell(i)),&
	\end{tikzcd}
	\]
	where
	$f$ is the surjection given in \cite[Proposition 1]{colliot1983torsion},
	$H^{i-1}(X_{\ov{k}_0},\mc{H}(\Q_\ell/\Z_\ell(i)))$ is the $E^{i-1,i}_2$ term of the Bloch-Ogus spectral sequence \cite{bloch1974gersten}, and $g$ is the edge homomorphism.	
	Hence the proof will follow once we show the anti-commutativity of the following diagram:
	\begin{equation}\label{d-bloch1}
	\begin{tikzcd}
	H^{i-1}(X_{\ov{k}_0},\mc{H}^i(\Q_\ell/\Z_\ell(i)))\ar[twoheadrightarrow, "f"]{r}\ar[d, "g"]& CH^i(X_{\ov{k}_0})\{\ell\}\ar[d,"\varinjlim_{k/k_0}\cl_k"]\\
	H^{2i-1}(X_{\ov{k}_0},\Q_\ell/\Z_\ell(i)) \ar[r, "\varinjlim_{k/k_0}\beta_k"]&\varinjlim_{k/k_0} H^{2i}(X_k,\Z_\ell(i)).
	\end{tikzcd}
	\end{equation}
	%Note that 
	Here $H^{i-1}(X_{\ov{k}_0},\mc{H}^i(\Q_\ell/\Z_\ell(i)))=\varinjlim_{k/k_0}H^{i-1}(X_{k},\mc{H}^i(\Q_{\ell}/\Z_{\ell}(i)))$, because the Gersten complex of $\mc{H}^i(\Q_\ell/\Z_\ell(i))$ on $X_{\ov{k}_0}$ is the direct limit of Gersten complexes on $X_k$. 
	%(reference).
	Hence the anti-commutativity of (\ref{d-bloch1}) is reduced to showing, 
	for every finite extension $k/k_0$ and every integer $\nu\geq 1$, the anti-commutativity of
	\begin{equation}\label{d-bloch2}
	\begin{tikzcd}
	H^{i-1}(X_{k},\mc{H}^i(\mu_{l^\nu}^{\otimes i}))\ar[twoheadrightarrow,"f"]{r}\ar[d, "g"]& CH^i(X_{k})[\ell^\nu]\ar[d,"\cl_k"]\\
	H^{2i-1}(X_{k},\mu_{l^\nu}^{\otimes i}) \ar[r, "\beta_k"]& H^{2i}(X_k,\Z_\ell(i)).
	\end{tikzcd}
	\end{equation}
	
	To prove that (\ref{d-bloch2}) anti-commutes,
	we proceed as in the proof of \cite[Proposition 1]{colliot1983torsion}.
	Recall that each element $\alpha\in H^{i-1}(X_{k},\mc{H}^i(\mu_{l^\nu}^{\otimes i}))$ is represented by a class $a\in H^{2i-1}_Z(X_k-Z',\mu_{l^\nu}^{\otimes i})$,
	%for some pair $(Z,Z')$ as above
	where $(Z,Z')$ is a pair of closed subsets of $X_k$ of codimension $i-1$ and $i$ respectively with $Z'\subset Z$, 
	that vanishes under the connecting homomorphism $H^{2i-1}_{Z-Z'}(X_k-Z',\mu_{l^\nu}^{\otimes i})\rightarrow H^{2i}_{Z'}(X_k,\mu_{l^\nu}^{\otimes i})$.
	We may now associate to the class $a$ two classes $b,c\in H^{2i}_Z(X_k,\Z_\ell(i))$ 
	whose images in $H^{2i}(X_{k},\Z_{\ell}(i))$ are $\beta_k\circ g(\alpha)$, $-\cl_k\circ f(\alpha)$ respectively.
	The argument is as follows, using the %commutative 
	diagram:
	%with exact rows and columns:
	\[
	\adjustbox{scale=0.9,center}{
	\begin{tikzcd}
	& & H^{2i}_{Z'}(X_k,\Z_\ell(i))\ar[r,"i"]\ar[d,"\ell^\nu"] &H^{2i}_Z(X_k,\Z_\ell(i))\\
	& H^{2i-1}_{Z-Z'}(X_k-Z',\Z_\ell(i))\ar[r,"\delta"]\ar[d,"p"] & H^{2i}_{Z'}(X_k,\Z_\ell(i))
	%\ar[d, "j_{2}"]
	&\\
	H^{2i-1}_Z(X_k,\mu_{\ell^\nu}^{\otimes i})\ar[r,"j"]\ar[d,"\beta_k"]&H^{2i-1}_{Z-Z'}(X_k-Z',\mu_{\ell^\nu}^{\otimes i}) %\ar[r,"\delta"]
	& %H^{2i}_{Z'}(X_k,\mu_{l^\nu}^{\otimes i})
	&\\
	H^{2i}_Z(X_k,\Z_\ell(i)).&&&
	\end{tikzcd}
	}
	\]
	Here the horizontal arrows are from the long exact sequences for cohomology with supports, and the vertical arrows are from the long exact sequences for $H^*_{Z'}(X_k,-)$, $H^*_{Z}(X_k,-)$, and $H^*_{Z-Z'}(X_k,-)$ induced by
	the short exact sequence of inverse systems of abelian sheaves on $X_{\text{\'et}}$:
	\begin{equation}\label{ses-invet}
	\begin{tikzcd}
	& \vdots \ar[d]& \vdots\ar[d] & \vdots \ar[d]&\\
	0\ar[r] &\mu_{\ell^{\nu'+1}}^{\otimes i} \ar[r]\ar[d,"\ell"]& \mu_{\ell^{\nu'+1+\nu}}^{\otimes i}\ar[r, "\ell^{\nu'+1}"]\ar[d, "\ell"]& \mu_{\ell^{\nu}}^{\otimes i}\ar[r]\ar[d,"\on{id}"] & 0\\
	0\ar[r] &\mu_{\ell^{\nu'}}^{\otimes i} \ar[r]\ar[d]& \mu_{\ell^{\nu'+\nu}}^{\otimes i}\ar[r, "\ell^{\nu'}"]\ar[d]& \mu_{\ell^{\nu}}^{\otimes i}\ar[r]\ar[d] & 0.\\
	& \vdots &\vdots &\vdots
	\end{tikzcd}
	\end{equation}
	By the choice on $a$, there exists $a_1\in H^{2i-1}_Z(X_k,\mu_{\ell^\nu}^{\otimes i})$ such that $j(a_1)=a$. 
	Set $b\coloneqq \beta_k(a_1)$.
	Meanwhile, after possibly %replacing 
	enlarging
	$Z'\subset Z$,
	$a$ lifts along $p$
	%, that is, 
	%there exists 
	to a class $a_2\in H^{2i-1}_{Z-Z'}(X_k-Z',\Z_\ell(i))$.
	%such that $p(a_2)=a$.
	Indeed, we may assume that: $Z-Z'$ is smooth, thus 
	\begin{align*}
	H^{2i-1}_{Z-Z'}(X_k-Z',\mu_{\ell^\nu}^{\otimes i})&=H^{1}(Z-Z', \mu_{\ell^{\nu}}),\\
	H^{2i-1}_{Z-Z'}(X_k-Z',\Z_\ell(i))&=H^{1}(Z-Z',\Z_{\ell}(1))
	\end{align*}
	by \cite[Theorem 3.17]{jannsen1988continuous};
	$a\in H^{1}(Z-Z', \mu_{\ell^{\nu}})$ lifts along the composition
	\[
	H^{0}(Z-Z',\G_{m})\xrightarrow{\Delta} H^{1}(Z-Z',\Z_{\ell}(1))\xrightarrow{p} H^{1}(Z-Z',\mu_{\ell^{\nu}}^{\otimes i}),
	\] 
	where $\Delta$ is the connecting homomorphism for the short exact sequence of inverse systems of abelian sheaves on $X_{\text{\'et}}$
	\[
	\begin{tikzcd}
	& \vdots \ar[d]& \vdots\ar[d] & \vdots \ar[d]&\\
	0\ar[r] &\mu_{\ell^{\nu'+1}} \ar[r]\ar[d,"\ell"]& \G_{m}\ar[r, "\ell^{\nu'+1}"]\ar[d, "\ell"]& \G_{m}\ar[r]\ar[d, "\on{id}"] & 0\\
	0\ar[r] &\mu_{\ell^{\nu'}} \ar[r]\ar[d]& \G_{m}\ar[r, "\ell^{\nu'}"]\ar[d]& \G_{m}\ar[r]\ar[d] & 0.\\
	& \vdots &\vdots &\vdots
	\end{tikzcd}
	\]
    (To see this, note that $p\circ\Delta$ at the direct limit over all $Z'\subset Z$ corresponds to the surjection $\oplus k(x)^{\times}\twoheadrightarrow \oplus k(x)^{\times}/\ell^{\nu}$, where the direct sums are over the generic points of $Z$.)
	Let $a_3=\delta(a_2)$.
	%, then $j_{2}(a_{3})=\delta(a)=0$, so there exits $a_4\in H^{2i}_{Z'}(X_k,\Z_\ell(i))$ such that $a_3=l^\nu a_4$.
	Then there exits $a_4\in H^{2i}_{Z'}(X_k,\Z_\ell(i))$ such that $a_3=l^\nu a_4$.
	Set $c\coloneqq i(a_4)$.
	It is now direct to see that $b,c$ satisfy the required properties.
	
	To complete the proof, it is enough to show that $b=c$.
	As the category of inverse systems of abelian sheaves on $X_{\text{\'et}}$ is an abelian category with enough injectives by \cite[Proposition 1.1]{jannsen1988continuous}, 
	we may take a Cartan-Eilenberg injective resolution of (\ref{ses-invet}).
	Now an argument analogous to \cite[p. 771]{colliot1983torsion} concludes the proof.
	\end{proof}
	
	\begin{proof}[Second Proof of \cref{complete-intersection2}]
	%Let $k_0$ be a field of characteristic different from $2$. 
	As in \Cref{section-ci}, let $G\coloneqq H\times \Z/2$, where $H$ is the Heisenberg group of order $32$, $Y\subset \P^N_{k_0}$ be a smooth complete intersection of dimension $15$ on which $G$ acts freely, and $X\coloneqq  Y/G$.
	By means of \cref{thm-blochjannsen} and the rigidity property of $\lambda$,
	%\cref{thm-blochjannsen}, 
	it is enough for us to show that $\lambda\colon CH^3(X_{F})\{2\}\rightarrow H^5(X_{F},\Q_2/\Z_2(3))$ is not injective for some algebraically closed field extension $F$ of a field of definition. 
	%of $X$.
	%$F/k_0$.
	
	%Since the Bloch map over the complex numbers coincides with the Deligne cycle class map on torsion by \cite[Proposition 3.7]{bloch1979torsion}, 
	The assertion in characteristic zero follows from \cite[Theorem 7.2]{totaro1997torsion}.
	In positive characteristic different from $2$, the assertion follows from \cite[Proposition 5.3 (b)]{quick2011torsion},
	because the group $H^5(X_{\ov{k}_0},\Z_2(3))$ is torsion by construction and the composition
	\[CH^2(X_{\ov{k}_0})\{2\}\xrightarrow{\lambda} H^5(X_{\ov{k}_0},\Q_2/\Z_2(3))\lhook\joinrel\xrightarrow{\beta_{\ov{k}_0}} H^6(X_{\ov{k}_0},\Z_2(3))\]
	coincides with the cycle class map.
	This concludes the proof.
	\end{proof}	
	
	We conclude this section by 
	%application of \cref{thm-blochjannsen} to 
	a remark on
	Schreieder's transcendental Abel-Jacobi map \cite[\S 7.5]{schreieder2022refined}.
	
	\begin{rem}\label{rem-taj}
	Suppose that $k_{0}$ is of finite type over its prime field.
	For every finite extension $k/k_{0}$,
	we have the transcendental Abel-Jacobi map:
	\[
	\lambda_{\text{tr},k}\colon 
	CH^{i}_{0}(X_{k})\{\ell\}
	\rightarrow H^{2i-1}(X_{k},\Q_{\ell}/\Z_{\ell}(i))/N^{i-1}H^{2i-1}(X_{k},\Q_{\ell}(i)),
	\]
	where $CH^{i}_{0}(X_{k})\{\ell\}$ is the kernel of $\cl_{k}\colon CH^{i}(X_{k})\{\ell\}\rightarrow H^{2i}(X_{k},\Z_{\ell}(i))$.
	Then one can observe that $\varinjlim_{k/k_0}\lambda_{\text{tr},k}=0$ by \cite[Proposition 7.16]{schreieder2022refined} and weight arguments.
	This shows that if
	\[\varinjlim_{k/k_{0}}CH^{i}_{0}(X_{k})\{\ell\}=\Ker\left(CH^{i}(X_{\ov{k}_{0}})\{\ell\}\xrightarrow{\varinjlim_{k/k_{0}}\cl_{k}}\varinjlim_{k/k_{0}}H^{2i}(X_{k},\Z_{\ell}(i))\right)\]
	is not zero, 
	or equivalently by \cref{thm-blochjannsen}, 
	if $\lambda$ is not injective over $\ov{k}_0$,
	then $\lambda_{\text{tr},k}$ is not injective for all sufficiently large finite extensions $k/k_{0}$ contained in $\ov{k}_0/k_0$.
	As described in \cref{rem-blochjannsen}, we already have several examples with the property, and we will give further such examples in Sections \ref{section-fourfold} and \ref{section-furtherex}.
	In \Cref{section-quadric}, we will provide examples with $\lambda_{\text{tr},k}\neq 0$.
	\end{rem}
	
\section{Proof of Theorem \ref{thm-fourfold}}\label{section-fourfold}
    Let $k_0$ be a number field.
    Let $B$ (resp. $E$) be an abelian threefold (resp. an elliptic curve) over $k_0$
    and set $A\coloneqq B\times E$.
    Suppose that $A$ has good ordinary reduction at some prime dividing $2$.
	For instance, one can take $k_0=\Q$ and $A$ to be the product of $4$ copies of the elliptic curve
	\[
	y^2+xy=x^3+1.
	\]
	
	Let $\iota$ be an involution acting on $B$ by $-1$ and $Y$ be the Kummer threefold associated to $B$, i.e., the blow up of $B/\iota$ at the $64$ singular points, so that $Y$ is smooth and contains $64$ disjoint copies of $\P^2$.
	Finally, set $X\coloneqq Y\times E$.
	Note that the action of $\iota$ lifts to $A$ where $\iota$ acts trivially on $E$,
	and $X$ can also be obtained by blowing up the quotient variety $A/\iota$ along the singular locus.
	
	In the following, we fix an embedding $k_0\hookrightarrow \C$.
	
	\begin{lem}\label{lem-uc4}
	$H^4_{\on{nr}}(X_\C,\Z/2)\neq 0$.
	\end{lem}
	\begin{proof}
	We follow the method of Diaz in \cite[\S 2.1]{diaz2020unramified}.
	In this proof,
	%For simplicity, 
	%By abuse of notation,
	we write $A, B, E, X$ for $A_\C, B_\C, E_\C, X_\C$.
	Letting $A^{\circ}\coloneqq A-(B[2]\times E)$, $U\coloneqq A^{\circ}/\iota$, and $\pi\colon A^\circ \rightarrow U$ be the quotient map,
	we have the following commutative diagram:
	\begin{equation}\label{d-diaz}
	\begin{tikzcd}
	H^4(A,\Z/2)\ar[r,"\sim"]\ar[d]& H^4(A^{\circ},\Z/2)\ar[d]&\ar[l, twoheadrightarrow, "\pi^*"'] H^4(U,\Z/2)\ar[rd]\ar[d]& \\
	H^4_{\on{nr}}(A,\Z/2)\ar[r,hook]& H^4_{\on{nr}}(A^{\circ},\Z/2)&\ar[l, "\pi^*"']H^4_{\nr}(U,\Z/2)&\ar[l, hook']H^4_{\on{nr}}(X,\Z/2).
	\end{tikzcd}
	\end{equation}
	Here the vertical arrows are the restriction maps and the horizontal arrows are the pull-back maps; the injectivity of $H^4_{\on{nr}}(A,\Z/2)\to H^4_{\on{nr}}(A^\circ,\Z/2)$ and $H^4_{\on{nr}}(X,\Z/2)\to H^4_{\on{nr}}(U,\Z/2)$ is by definition of unramified cohomology; the map
	$H^4(A,\Z/2)\to H^4(A^{\circ},\Z/2)$ is an isomorphism because $\on{codim}(B[2]\times E,A)=3$.
	
	We need to check that (i) $\pi^*\colon H^4(U,\Z/2)\twoheadrightarrow H^4(A^\circ,\Z/2)$ and (ii) $H^4(U,\Z/2)\rightarrow H^4_{\on{nr}}(U,\Z/2)$ factors through $H^4_{\on{nr}}(X,\Z/2)$.
	As for (i), note that
	$A^{\circ}=(B-B[2])\times E$ and
	$U=(B-B[2])/\iota\times E$.
	Letting $\rho\colon B-B[2]\rightarrow (B-B[2])/\iota$ be the quotient map,
	it is enough for us to show that $\rho^*\colon H^i((B-B[2])/\iota,\Z/2)\rightarrow H^i(B-B[2],\Z/2)$ is surjective for $i=2, 3,4$.
	Since $\on{codim}(B[2],B)=3$, 
	the restriction map
	\[
	\Lambda^iH^1(B,\Z/2)\xrightarrow{\sim}H^i(B,\Z/2)\rightarrow H^i(B-B[2],\Z/2)
	\]
	is an isomorphism for $i\leq 4$.
	So it suffices to show that $\rho^*\colon H^1((B-B[2])/\iota,\Z/2)\rightarrow H^1(B-B[2],\Z/2)$ is surjective,
	which follows from the fact that the short exact sequence
	\[
	1\rightarrow \pi_1(B-B[2])\rightarrow \pi_1((B-B[2])/\iota)\rightarrow \{\pm 1\}\rightarrow 1
	\]
	splits.
 % on the right. 
	Here the splitting is given by the non-trivial element in the fundamental group of $\mathbb{RP}^5$ that appears as the quotient of the boundary $\mathbb{S}^5$ of an open ball neighborhood of a $2$-torsion point in $B$, as observed in the first paragraph of the proof of \cite[Theorem 1]{spanier1956homology} (see also \cite[p. 267]{diaz2020unramified}).
	Alternatively, (i) directly follows from \cite[Corollary 2.8]{diaz2020unramified}, because the assumptions for the statement are satisfied:
	$B[2]\times E$ is smooth, $\on{codim}(B[2]\times E,A)=3$, $\iota$ acts by $-1$ on the normal bundle $N_{B[2]\times E/A}$, and $\iota$ acts trivially on $H^1(A,\Z/2)$.
	As for (ii), the direct computation of the unramified cohomology group using the Gersten complex reduces it to the vanishing $H^3_{\nr}(X-U,\Z/2)=0$ (see \cite[Lemma 2.10]{diaz2020unramified}).
	The vanishing indeed holds because $X-U$ is $64$ disjoint copies of $\P^2\times E$ and \[H^3_{\on{nr}}(\P^2\times E,\Z/2)=H^3_{\on{nr}}(E,\Z/2)=0.\]

	Finally, a theorem of Bloch--Esnault \cite[Theorem 1.2]{bloch1996coniveau} shows that $H^4(A,\Z/2)\rightarrow H^4_{\on{nr}}(A,\Z/2)$ is non-zero. (Here we use the rigidity property for unramified cohomology with torsion coefficients \cite[Theorem 4.4.1]{colliot1995birational}.)
	This, with (\ref{d-diaz}), concludes the proof.
	\end{proof}
	
	\begin{prop}\label{prop-deligne}
	$\cl_{\mc{D}}\colon CH^3(X_\C)\{2\}\rightarrow H^6_{\mc{D}}(X_\C,\Z(3))$ is not injective.
	\end{prop}
	\begin{proof}
	One needs to relate the fourth unramified cohomology group to the kernel of the Deligne cycle class map on torsion in codimension $3$.
	We start with a short exact sequence given by \cite[Theorem 0.2]{voisin2012degree4} and \cite[Remark 4.2 (1)]{ma2017torsion}:
	\begin{align}\label{ses-unramified}
	0\rightarrow \Lambda^5(X_\C)_{\tors} \rightarrow H^4_{\nr}(X_\C,\Q/\Z)/H^4_{\nr}(X_\C,\Z)\otimes \Q/\Z\rightarrow \mc{T}^3(X_\C)\rightarrow 0,
	\end{align}
	where
	\begin{align*}
	\Lambda^5(X_\C)&\coloneqq H^5(X_\C,\Z)/N^2H^5(X_\C,\Z),\\ 
	\mc{T}^3(X_\C)&\coloneqq \Ker\left(\cl_{\mc{D}}\colon CH^3(X_\C)_{\tors}\rightarrow H^6_{\mc{D}}(X_\C,\Z(3))\right)/\alg.
	\end{align*}
	(The notation $/\alg$ in the above means quotient by the algebraically trivial cycles in the kernel.)
	It is important for us that $CH_0(X_\C)$ is supported in dimension $\leq 3$, because $CH_0(Y_\C)$ is supported in dimension $\leq 2$ by \cite[\S 4 (1)]{bloch1983remark}.
	By decomposition of the diagonal and the Bloch--Kato conjecture proved by Voevodsky, we have 
	\begin{align}\label{eq-unramified}
	H^4_{\nr}(X_\C,\Z)=0
	\end{align}
	(see \cite[Proposition 3.3 (i)]{colliot2012cohomologie}).
	Moreover, \cite[Theorem 1.1]{suzuki2020remark} yields
	\begin{align*}
	&\Coker\left(H^5(X_\C,\Z)_{\tors}\rightarrow \Lambda^5(X_\C)_{\tors}\right)\\
	&\simeq
	\Ker\left(\cl_{\mc{D}}\colon CH^3(X_\C)_{\alg,\tors}\rightarrow H^6_{\mc{D}}(X_\C,\Z(3)) \right),
	\end{align*}
	where we write $CH^3(X_\C)_{\alg,\tors}\subset CH^3(X_\C)$ for the subgroup of algebraically trivial torsion cycles.
	Note that $H^5(X_\C,\Z)$ is in fact torsion-free, because $Y_\C$ and $E_\C$ have torsion-free cohomology (use \cite[Theorem 2]{spanier1956homology} for the Kummer threefold $Y_\C$),
	hence
	\begin{align}\label{eq-lambda}
	\Lambda^5(X_\C)_{\tors}\simeq\Ker\left(\cl_{\mc{D}}\colon CH^3(X_\C)_{\alg,\tors}\rightarrow H^6_{\mc{D}}(X_\C,\Z(3)) \right).
	\end{align}
	By (\ref{ses-unramified}), (\ref{eq-unramified}), and (\ref{eq-lambda}), it remains to show that $H^4_{\nr}(X_\C,\Q/\Z)\{2\}\neq 0$.
	This can be deduced from \cref{lem-uc4}, because the natural map
	\[
	H^4_{\on{nr}}(X_\C,\Z/2)\rightarrow H^4_{\on{nr}}(X_\C,\Q/\Z) 
	\]
	is injective again by the Bloch--Kato conjecture (see \cite[Theorem 1.1]{auel2017universal}).
	The proof is now complete.
	\end{proof}
	
	We prove a strengthened version of \cref{thm-fourfold}.
	
	\begin{thm}\label{thm-fourfold2}
	Let $k_0$ be a field 
	%of finite type over $\Q$.
	of characteristic zero.
	Then there exist a fourfold product $X=Y\times E$ over $k_0$, where $Y$ is a Kummer threefold and $E$ is an elliptic curve, and a finite extension $k/k_0$ such that the cycle class map
	\[
	\cl\colon CH^3(X_k)[2]\rightarrow H^6(X_k,\Z_2(3))
	\]
	is not injective.
	\end{thm}
	\begin{proof}
%	We first assume that $k_0$ is a number field.
%	Choose any subfield $\widetilde{k}_{0}\subset k_{0}$ that is finite over $\Q$.
	Let 
	$X=Y\times E$ be a fourfold product 
	over a subfield $\widetilde{k}_{0}\subset k_{0}$ that is finite over $\Q$,
	%over $\widetilde{k}_{0}$
	%over $k_0$
	%over $\Q$
	as given at the beginning of this section.
	%constructed 
	%as at the beginning of this section.
	Fixing an embedding $\widetilde{k}_{0}\hookrightarrow \C$,
%	Since the Deligne cycle class map on torsion coincides with Bloch's map over the complex numbers,
	% by \cite[Proposition 3.7]{bloch1979torsion},
	\cref{prop-deligne} shows that \[\lambda\colon CH^3(X_{\C})\{2\}\rightarrow H^5(X_{\C},\Q_2/\Z_2(3))\] 
	is not injective,
	hence by the rigidity property of $\lambda$, the same result holds over $\ov{\widetilde{k}}_0$, then over $\ov{k}_0$.
	%Because of the rigidity properties,
	%the non-injectivity also holds over $\ov{\Q}$, hence over $\ov{k}_0$ for an arbitrary field $k_0$ of characteristic zero.
	%Fixing an embedding $k_0\hookrightarrow \C$,
	\cref{thm-blochjannsen}
	%and the rigidity property of $\lambda$ 
	%then implies 
	%then show
	now shows
	that there exists a finite extension $k/k_0$ such that 
	\[\cl\colon CH^3(X_{k})\{2\}\rightarrow H^6(X_{k},\Z_2(3))\]
	is not injective.
%	The proof is complete.	
%	Finally, 
%	the assertion for an arbitrary field $k_0$ 
%	of finite type over $\Q$
%	of characteristic zero 
%	follows from the rigidity property of $\lambda$, using the above example $X$.
	This finishes the proof.
	\end{proof}
 
	\section{Further examples in codimension three}\label{section-furtherex}
	In this section, we provide further counterexamples to \cref{schreieder} in codimension $3$.
	By \cref{thm-blochjannsen}, this is reduced to finding examples for which Bloch's map $\lambda$ is not injective over some algebraically closed field extension of a field of definition.
	To achieve this, we use non-torsion type counterexamples to the integral Hodge and Tate conjectures, inspired by the work of Soul\'e--Voisin \cite{soule2005torsion}.
	
	Let $k_0$ be a field,
	$\ell$ be a prime number invertible in $k_0$, $i\geq 0$ be an integer, and $Y$ be a smooth projective variety over $k_0$.
	We define 
	\[
	\widetilde{Z}^{2i}_{\text{\'et},\ell}(Y_{(k_{0})_s})\coloneqq
	\Coker\left(H^{2i}(Y_{(k_{0})_s},\Z_\ell(i))_{\tors}\rightarrow H^{2i}(Y_{(k_{0})_s},\Z_\ell(i))^{(1)}/H^{2i}_{\alg}(Y_{(k_{0})_s},\Z_\ell(i))\right),
	\]
	where $H^{2i}(Y_{(k_{0})_s},\Z_\ell(i))^{(1)}\subset H^{2i}(Y_{(k_{0})_s},\Z_\ell(i))$ is the $\Gal((k_{0})_s/k_0)$-submodule consisting of elements with open stabilizer, 
	and $H^{2i}_{\alg}(Y_{(k_{0})_s},\Z_\ell(i))$ is the image of the cycle class map $\cl\colon CH^i(Y_{(k_{0})_s})\otimes_{\Z} \Z_\ell\rightarrow H^{2i}(Y_{(k_{0})_s},\Z_\ell(i))$.
	The group $\widetilde{Z}^{2i}_{\text{\'et},\ell}(Y_{(k_{0})_s})$ is well-defined
	because $H^{2i}(Y_{(k_{0})_s},\Z_\ell(i))^{(1)}\subset H^{2i}(Y_{(k_{0})_s},\Z_\ell(i))$ is saturated by \cite[Lemma 4.1]{colliot2013cycles}.
	Note that
	$\widetilde{Z}^{2i}_{\text{\'et},\ell}(Y_{(k_{0})_s})_{\tors}=0$ if and only if the sublattice
	\[
	H^{2i}_{\alg}(Y_{(k_{0})_s},\Z_\ell(i))_{\tf}\subset H^{2i}(Y_{(k_{0})_s},\Z_\ell(i))^{(1)}_{\tf}
	\]
	is saturated.
	When $k\subset \C$, we similarly define
	\[
	\widetilde{Z}^{2i}(Y_\C)\coloneqq  \Coker\left(H^{2i}(Y_\C,\Z)_{\tors}\rightarrow  \on{Hdg}^{2i}(Y_\C,\Z)/H^{2i}_{\alg}(Y_\C,\Z)\right),
	\]
	where $\on{Hdg}^{2i}(Y_\C,\Z)\subset H^{2i}(Y_\C,\Z)$ is the subgroup of integral Hodge classes and $H^{2i}_{\alg}(Y_\C,\Z)\coloneqq \Image\left(\cl\colon CH^i(Y_\C)\rightarrow H^{2i}(Y_\C,\Z)\right)$.
	Note that $\widetilde{Z}^{2i}(Y_\C)_{\tors}=0$ if and only if the sublattice 
	\[H^{2i}_{\alg}(Y_\C,\Z)_{\tf}\subset \on{Hdg}^{2i}(Y_\C,\Z)_{\tf}\] 
	is saturated.
	
	\begin{lem}\label{lem-tatehodgetobloch}
	With the same notation as above, suppose either:
	$\widetilde{Z}^{2i}_{\text{\'et},\ell}(Y_{(k_{0})_s})\{\ell\}\neq 0$, or $k_0\subset \C$ and $\widetilde{Z}^{2i}(Y_\C)\{\ell\}\neq 0$.
	Then there exist a finitely generated extension $K_0/k_0$ with $\on{tr.deg}_{k_0} K_0= 1$ and an elliptic curve $E$ over $K_0$ such that, letting $X\coloneqq Y\times_{k_0} E$, the map $\lambda\colon CH^{i+1}(X_{\ov{K}_0})\{\ell\}\rightarrow H^{2i+1}(X_{\ov{K}_0},\Q_\ell/\Z_\ell(i+1))$ is not injective.
	\end{lem}
	\begin{proof}
	We only do the first case; the second case is similar (also see \cite[Proposition 3.1]{suzuki2020remark}).
	%Tensored with $\Q_\ell/\Z_\ell$ 
	After tensor $\Q_{\ell}/\Z_{\ell}$,
	the short exact sequence
	\[
	0\rightarrow H^{2i}_{\alg}(Y_{(k_{0})_s},\Z_\ell(i))\rightarrow H^{2i}(Y_{(k_{0})_s},\Z_\ell(i))^{(1)}\rightarrow H^{2i}(Y_{(k_{0})_s},\Z_\ell(i))^{(1)}/H^{2i}_{\alg}(Y_{(k_{0})_s},\Z_\ell(i))\rightarrow 0
	\]
	yields an exact sequence
	\[
	0\rightarrow \widetilde{Z}^{2i}_{\text{\'et},\ell}(Y_{(k_{0})_s})\{\ell\}
	\rightarrow H^{2i}_{\alg}(Y_{(k_{0})_s},\Z_\ell(i))\otimes \Q_\ell/\Z_\ell\rightarrow H^{2i}(Y_{(k_{0})_s},\Z_\ell(i))\otimes \Q_\ell/\Z_\ell.
	\]
	From the assumption, we now see that there exists a non-zero $\alpha\in CH^i(Y_{(k_{0})_s})\otimes \Q_\ell/\Z_\ell$ that vanishes 
	in $H^{2i}(Y_{(k_{0})_s},\Z_\ell(i))\otimes \Q_\ell/\Z_\ell$.
	Note that, by passing to the algebraic closure, we get isomorphisms
	\[
	CH^i(X_{(k_{0})_s})\otimes \Q_\ell/\Z_\ell\xrightarrow{\sim}CH^i(X_{\ov{k}_0})\otimes\Q_\ell/\Z_\ell,\,H^{2i}(X_{(k_{0})_s},\Z_\ell(i))\xrightarrow{\sim}H^{2i}(X_{\ov{k}_0},\Z_\ell(i)).
	\]
	%where the former follows from \cite[Lemma 6.8]{schreieder2022refined}.
	Let $\alpha'\in CH^i(X_{\ov{k}_0})\otimes \Q_\ell/\Z_\ell$ be the image of $\alpha$.
	
	Let $K_0/k_0$ be a finitely generated field extension with $\on{tr.deg}_{k_0} K_0=1$ and $E$ be an elliptic curve over $K_0$ with $j(E)\not\in \ov{k}_0$.
	Fixing a component $\Q_\ell/\Z_\ell$ of $CH^1(E_{\ov{K}_0})\{\ell\}=(\Q_\ell/\Z_\ell)^2$, we indentify $\alpha'$ with an element in $CH^{i}(Y_{\ov{k}_0})\otimes CH^1(E_{\ov{K}_0})\{\ell\}$.
	Letting $X\coloneqq Y\times_{k_0} E$,
	a theorem of Schoen \cite[Theorem 0.2]{schoen2000certain} shows that the image $\beta$ of $\alpha'$ under the exterior product map
	\[CH^i(Y_{\ov{k}_0})\otimes CH^1(E_{\ov{K}_0})\{\ell\}\xrightarrow{\times} CH^{i+1}(X_{\ov{K}_0})\{\ell\}\]
	is non-zero.
	Now it remains for us to show that $\beta\in CH^{i+1}(X_{\ov{K}_0})\{\ell\}$ is in the kernel of $\lambda$.
	This follows from the commutative diagram:
	\[
	\begin{tikzcd}
	CH^i(Y_{\ov{k}_0})\otimes CH^1(E_{\ov{K}_0})\{\ell\}\ar[r,"\cl\otimes\lambda"]\ar[d,"\times"]& H^{2i}(Y_{\ov{k}_0},\Z_\ell(i))\otimes H^1(E_{\ov{K}_0},\Q_\ell/\Z_\ell(1))\ar[d,"\cup"]\\
	CH^{i+1}(X_{\ov{K}_0})\ar[r,"\lambda"]& H^{2i+1}(X_{\ov{K}_0},\Q_\ell/\Z_\ell(i+1)).
	\end{tikzcd}
	\]
	The proof is complete.
	\end{proof}
	
%	Similarly, one can show:
%	\begin{lem}\label{lem-modltobloch}
%	With the same notation as above, suppose that $CH^i(Y_{\ov{k}_0})/\ell$ is infinite.
%	Then there exists a finitely generated extension $K_0/k_0$ with $\on{tr.deg}_{k_0} K_0=1$ and an elliptic curve $E$ over $K_0$ such that, letting $X\coloneqq Y\times_{k_0} E$, the map $\lambda\colon CH^{i+1}(X_{\ov{K}_0})\{\ell\}\rightarrow H^{2i+1}(X_{\ov{K}_0},\Q_\ell/\Z_\ell(i))$ has kernel with infinite $\ell$-torsion elements.
%	\end{lem}

	\cref{lem-tatehodgetobloch} can be applied to non-torsion type counterexamples to the integral Hodge conjecture \cite{kollar1992trento,colliot2012cohomologie,totaro2013integral,diaz2020unramified,ottem2020pencil} or the integral Tate conjecture \cite{totaro2013integral,pirutka2015note}.
	One may take $k_0=\Q$ for the examples in \cite{kollar1992trento,colliot2012cohomologie,totaro2013integral,diaz2020unramified,ottem2020pencil} and $k_0$ to be a finite field for the examples in \cite{pirutka2015note}.
%	\cref{lem-modltobloch} can be applied to abelian threefolds in \cite{totaro2016complex}.
	
	\cref{thm-blochjannsen} then produces various examples of fields $K$ of finite type over the prime fields of transcendence degree $1$, prime numbers $\ell$ invertible in $K$, and smooth projective $K$-varieties $X$ such that
	$\cl\colon CH^3(X)[\ell]\rightarrow H^6(X,\Z_\ell(3))$
	is not injective.
	Those with the best bounds are: 
	fourfolds in characteristic zero;
	%with $\on{tr.deg}_\Q K=1$;  
	eightfolds in positive characteristic.
	%characteristic $p>0$ with $\on{tr.deg}_{\F_p}K=1$.
	
\section{Proof of Theorem \ref{quadric}}\label{section-quadric}

	\begin{lem}\label{tate-module}
		Let $k$ be a field and $\ell$ be a prime invertible in $k$. Then $H^2(k,\Z_{\ell}(1))\simeq T_{\ell}(\on{Br}(k))$. In particular, $H^2(k,\Z_{\ell}(1))$ is torsion-free.
	\end{lem}
	
	\begin{proof}
		By \cite[Theorem 2.7.5]{neukirch2008cohomology}, we have a short exact sequence
		\[0\to {\varprojlim_m}^1H^1(k,\mu_{\ell^m})\to H^2(k,\Z_{\ell}(1))\to \varprojlim_m H^2(k,\mu_{\ell^m})\to 0.\]
		The Kummer sequence \[1\to \mu_{\ell^m}\to \Gm\to \Gm\to 1\] gives natural identifications \[H^1(k,\mu_{\ell^m})=k^{\times}/k^{\times \ell^m},\qquad H^2(k,\mu_{\ell^m})=\on{Br}(k)[\ell^m].\] The induced maps $k^{\times}/k^{\times \ell^{m+1}}\to k^{\times}/k^{\times \ell^m}$ are the natural quotient maps, and in particular they are surjective. It follows that the sequence of the $H^1(k,\mu_{\ell^m})$ satisfies the Mittag-Leffler condition and so ${\varprojlim}^1H^1(k,\mu_{\ell^m})=0$. The induced maps $\on{Br}(k)[\ell^{m+1}]\to \on{Br}(k)[\ell^m]$ are given by multiplication by $\ell$, hence $\varprojlim H^2(k,\mu_{\ell^m})=T_{\ell}(\on{Br}(k))$.
	\end{proof}

\begin{lem}\label{h4-vanishes-kt}
    Let $k$ be a global field, $\ell$ be a prime number invertible in $k$, and $k(t)/k$ be a purely transcendental extension of transcendence degree $1$. If $\ell=2$, suppose that $k$ is totally imaginary. Then $H^4(k(t),\Z_{\ell}(2))=0$.
\end{lem}

\begin{proof}
%By \cite[(3.1)]{jannsen1988continuous},
By \cite[Theorem 2.7.5]{neukirch2008cohomology}, 
we have a short exact sequence
\begin{equation}\label{lim1}0\to {\varprojlim_n}^1H^3(k(t),\mu_{\ell^n}^{\otimes 2})\to H^4(k(t),\Z_{\ell}(2))\to \varprojlim_n H^4(k(t),\mu_{\ell^n}^{\otimes 2})\to 0.\end{equation}
By \cite[II.4.4, Proposition 13]{serre1997galois} we have $\cd_{\ell}(k)\leq 2$, and so \cite[II.4.2, Proposition 11]{serre1997galois} implies $\cd_\ell(k(t))\leq 3$. It follows that the group $H^4(k(t),\mu_{\ell^n}^{\otimes 2})$ is trivial for all $n\geq 0$, hence $\varprojlim_n H^4(k(t),\mu_{\ell^n}^{\otimes 2})=0$. In view of (\ref{lim1}), the proof will be complete once we show that $\varprojlim_n^1H^3(k(t),\mu_{\ell^n}^{\otimes 2})=0$.

We regard $k(t)$ as the function field of $\P^1_k$. By \cite[p. 113]{serre1997galois} we have an exact sequence
\[0\to H^3(k,\mu_{\ell^n}^{\otimes 2})\to H^3(k(t),\mu_{\ell^n}^{\otimes 2}) \to \oplus_{x\in (\P^1_k)^{(1)}}H^2(k(x),\mu_{\ell^n})\xrightarrow{C} H^2(k,\mu_{\ell^n})\to 0\]
which is functorial in $n\geq 0$. Since $\cd_{\ell}(k)\leq 2$, the first term $H^3(k,\mu_{\ell^n}^{\otimes 2})$ vanishes. The surjective map $C$ is the direct sum of the corestriction maps along the field extensions $k(x)/k$, and so the point at infinity $\infty \in \P^1_k$ determines a section of $C$. We obtain a decomposition
\begin{equation}\label{residues-splits}H^3(k(t),\mu_{\ell^n}^{\otimes 2})\simeq\oplus_{x\in (\A^1_k)^{(1)}}H^2(k(x),\mu_{\ell^n})\simeq\oplus_{x\in (\A^1_k)^{(1)}}\on{Br}(k(x))[\ell^n].\end{equation}
 The isomorphism on the right comes from the Kummer short exact sequence. The isomorphism (\ref{residues-splits}) is functorial in $n$, where on the right the transition maps $\on{Br}(k(x))[\ell^{n+1}]\to \on{Br}(k(x))[\ell^n]$ are given by multiplication by $\ell$.

Suppose first that $k$ is totally imaginary. Then for every closed point $x$ of $\A^1_k$ the residue field $k(x)$ is also totally imaginary. It follows from the celebrated theorem of Albert, Brauer, Hasse and Noether \cite[Theorem 8.1.17]{neukirch2008cohomology} that  $\on{Br}(k(x))$ is divisible. Thus the maps $\on{Br}(k(x))[\ell^{n+1}]\to \on{Br}(k(x))[\ell^n]$ given by multiplication by $\ell$ are surjective, hence by (\ref{residues-splits}) so are the transition maps $H^3(k(t),\mu_{\ell^{n+1}}^{\otimes 2})\to H^3(k(t),\mu_{\ell^n}^{\otimes 2})$. This shows that the inverse system $\{H^3(k(t),\mu_{\ell^n}^{\otimes 2})\}_{n\geq 0}$ satisfies the Mittag-Leffler condition, and so $\varprojlim_n^1H^3(k(t),\mu_{\ell^n}^{\otimes 2})=0$ by \cite[Proposition 2.7.4]{neukirch2008cohomology}, as desired. 

Suppose now that $k$ is not totally imaginary. Then under our assumptions $\ell\neq 2$. By \cite[Theorem 8.1.17]{neukirch2008cohomology}, the group $\on{Br}(k(x))$ is the direct sum of a divisible group and a finite elementary $2$-group. Then, since $\ell$ is odd, the maps $\on{Br}(k(x))[\ell^{n+1}]\to \on{Br}(k(x))[\ell^n]$ given by multiplication by $\ell$ are surjective, and the conclusion follows as in the totally imaginary case.
\end{proof}

\cref{quadric} is a special case of the following more general statement.

\begin{thm}\label{quadric2}
Let $k$ be a global field, $k(t)$ be a purely transcendental extension of $k$ of transcendence degree $1$, and $\ell$ be a prime invertible in $k$. If $\ell=2$, suppose that $k$ is totally imaginary, and if $\ell$ is odd suppose that $\on{char}(k)=0$. Then there exists a norm variety $X$ of dimension $\ell^2-1$ over $k(t)$ such that \[\cl\colon CH^2(X)[\ell]\to H^4(X,\Z_\ell(2))\]
is not injective.
\end{thm}
\begin{proof}
%By \cite[II.4.4 Proposition 13]{serre1997galois} we have $\on{cd}_{\ell}(k)\leq 2$, hence $H^3(k,\mu_{\ell}^{\otimes 2})=0$. Thus, 
By (\ref{residues-splits}) and the theorem of Albert, Brauer, Hasse
and Noether \cite[Theorem 8.1.17]{neukirch2008cohomology}, we have $H^3(k(t),\mu_{\ell}^{\otimes 2})\neq 0$. Let $X$ be a norm variety associated to a non-trivial symbol $s\in H^3(k(t),\mu_{\ell}^{\otimes 2})$, as constructed by Rost \cite{suslin2006norm}; see also \cite[\S 5d]{karpenko2013standard}. The $k$-variety $X$ is smooth projective of dimension $\ell^2-1$. The pure Chow motive with $\Z_{(\ell)}$-coefficients $M(X;\Z_{(\ell)})$ of $X$ contains the Rost motive $\mc{R}$ of $s$ as a direct summand. By \cite[Theorem RM.10]{karpenko2013standard}, we have $CH^2(\mc{R})=\Z/\ell$, hence $CH^2(X)[\ell]\neq 0$. (We apply \cite[Theorem RM.10]{karpenko2013standard} with $p=\ell$, $n=2$, $k=1$ and $i=1$. By definition $b=1+p$, hence $j=bk-p^i+1=2$.) Let $\alpha\in CH^2(\mc{R})[\ell]$ be a non-zero element.

If $\ell=2$, we may construct $X$ and $\alpha$ in any characteristic different from $2$ as follows. Let $\mc{O}$ be the ring of integers of $k$, $\pi\in \mc{O}$ be a prime element, and $u\in \mc{O}$ be such that the class of $u$ in the residue field $\mc{O}/\pi$ is not a square. The quadratic form \[q_0\coloneqq \ang{1,-u}\otimes\ang{1,-\pi}=\ang{1,-u,-\pi,u\pi}\] over $k$ is the norm form for the quaternion algebra $(u,\pi)$, hence it is anisotropic over $k$. By \cite[VI. Proposition 1.9]{lam2005introduction}, the quadratic form
\[q\coloneqq q_0\perp t\ang{1}=\ang{1,-u,-\pi,u\pi,t}\] is anisotropic over $k(\!(t)\!)$, hence over $k(t)$. Let $X\subset \P^4_{k(t)}$ be the smooth projective quadric hypersurface over $k(t)$ defined by $q=0$. By \cite[Theorem 5.3]{karpenko1990algebro} we have $CH^2(X)_{\on{tors}}\simeq\Z/2$. (In the notation of \cite[p. 120]{karpenko1990algebro}, $q=\ang{\!\ang{u,\pi}\!}\perp\ang{t}$.) We let $\alpha\in CH^2(X)_{\on{tors}}$ be the generator. The quadratic form $q$ is a neighbor of the Pfister form $\ang{\!\ang{u,\pi,-t}\!}$, hence $X$ is a norm variety for the symbol $(u)\cup(\pi)\cup (-t)\in H^3(k(t),\Z/2)$.

\medskip

We are going to prove that $\cl$ is not injective in codimension $2$ by showing that $\cl(\alpha)=0$ in $H^4(X,\Z_{\ell}(2))$. 
Consider the Hochschild-Serre spectral sequence in continuous $\ell$-adic cohomology %\cite[Corollary (3.4)]{jannsen1988continuous}
		\begin{equation}\label{hochschild-serre-quadric}E_2^{i,j}=H^i(k(t),H^j(X_{k(t)_s},\Z_\ell(2)))\Rightarrow H^{i+j}(X,\Z_\ell(2)).\end{equation}
		It yields a filtration
		\[\{0\}=F^5\subset F^4\subset\cdots \subset F^1\subset F^0=H^4(X,\Z_\ell(2))\]
		where $F^i/F^{i+1}$ is a subquotient (resp. a submodule) of $H^i(k,H^{4-i}(X_{k(t)_s},\Z_\ell(2)))$ for all $0\leq i\leq 4$ (resp. for $i=0,1$). Let $\rho\colon M(X;\Z_{(\ell)})\to M(X;\Z_{(\ell)})$ be the projector onto the direct summand $\mc{R}$, so that $\alpha \in \rho^*CH^2(X)$. (When $\ell=2$ and $X$ is the quadric described above, we could also take $\rho=\on{id}$ in what follows.) Since the Hochschild-Serre spectral sequence is natural with respect to correspondences, $\rho$ and $1-\rho$ respect $F^{\cdot}$ and determine a direct sum decomposition $F^{\cdot}=\rho^*F^{\cdot}\oplus(1-\rho^*)F^{\cdot}$, where $\rho^*F^{i}/\rho^*F^{i+1}$ is a subquotient (resp. a submodule) of $H^i(k,\rho^*H^{4-i}(X_{k(t)_s},\Z_\ell(2)))$ for all $0\leq i\leq 4$ (resp. for $i=0,1$).
	
	The Rost motive $\mc{R}_{k(t)_s}$ is a finite direct sum of powers of the Tate motive. Thus, for all $j\geq 0$ we have $\rho^*H^{2j+1}(X_{k(t)_s},\Z_{\ell})=0$ and $\rho^*H^{2j}(X_{k(t)_s},\Z_{\ell})\simeq \Z_{\ell}(-j)^{\oplus r_{2j}}$ for some integers $r_{2j}\geq 0$. It follows that \[H^1(k(t),\rho^*H^3(X_{k(t)_s},\Z_{\ell}(2)))=H^3(k(t),\rho^*H^1(X_{k(t)_s},\Z_{\ell}(2)))=0.\] 
	Since $H^0(X_{k(t)_s},\Z_{\ell}(2))\simeq \Z_{\ell}(2)$, the direct summand $\rho^*H^0(X_{k(t)_s},\Z_{\ell}(2))$ is either $0$ or $\Z_{\ell}(2)$. (As $CH^0(\mc{R})=\Z_{(\ell)}$ by \cite[Theorem RM.10]{karpenko2013standard}, we actually have $\rho^*H^0(X_{k(t)_s},\Z_{\ell}(2))=\Z_{\ell}(2)$.) Thus, by \Cref{h4-vanishes-kt}, \[H^4(k(t),\rho^*H^0(X_{k(t)_s},\Z_{\ell}(2)))=0.\]
	We deduce that $\rho^*F^1=\rho^*F^2$ and $\rho^*F^3=\rho^*F^4=\rho^*F^5=0$. Therefore $\rho^*F^1=\rho^*F^2/\rho^*F^3$, that is, we have an exact sequence
    \begin{equation}\label{f2f3}
    0\to \rho^*F^2/\rho^*F^3\to \rho^*H^4(X,\Z_\ell(2))\to \rho^*H^4(X_{k(t)_s},\Z_\ell(2)).
    \end{equation}
	We know that $\rho^*H^4(X_{k(t)_s},\Z_\ell(2))\simeq \Z_\ell^{\oplus r_4}$ is torsion-free. By \Cref{tate-module} the group
    \[H^2(k(t),\rho^*H^2(X_{k(t)_s},\Z_{\ell}(2)))\simeq T_\ell(\on{Br}(k(t)))^{\oplus r_2}\]
     is also torsion-free. By \cite[p. 262 and footnote 3]{jannsen2000letter} and \cite{ekedahl1990adic} (see also the announcement in \cite[Remark 6.15(b)]{jannsen1988continuous}), all differentials in (\ref{hochschild-serre-quadric}) are torsion, hence $\rho^*F^2/\rho^*F^3$ is torsion-free. Now (\ref{f2f3}) implies that $\rho^*H^4(X,\Z_\ell(2))$ is torsion-free. Since $\cl(\alpha)\in \rho^*H^4(X,\Z_\ell(2))$ and $\ell\cl(\alpha)=0$, we conclude that $\cl(\alpha)=0$.
\end{proof}

\begin{rem}[Colliot-Th\'el\`ene]
    We sketch a more direct proof of the fact, used in the proof of \Cref{quadric2}, that the group $H^3(k(t),\mu_{\ell}^{\otimes 2})$ is non-zero. We first note that if a symbol $(a,b)\in \on{Br}(k)[\ell]=H^2(k,\mu_\ell)$ is non-zero, then the residue of $(a,b,t)\in H^3(k(t),\mu_{\ell}^{\otimes 2})$ is non-zero, hence $(a,b,t)\neq 0$. Therefore, it suffices to show that $\on{Br}(k)[\ell]\neq 0$ for all global fields $k$. 
    
    One can show that $\on{Br}(k)[2]\neq 0$ by constructing a conic $X^2-aY^2-bT^2=0$ over $k$ without rational points. If $\ell$ is odd, one can construct a non-zero element of $\on{Br}(k)[\ell]$ by taking a cyclic extension $K/k$ of degree $\ell$, a place $v$ where $K/k$ is inert (using the Chebotarev Density Theorem), an element $c\in k_v^\times$ which is not a norm from $K_v^\times$, and approximating $c$ by an element of $k^\times$.
\end{rem}

\begin{rem}
    %Our example for \Cref{quadric2} when $\ell=2$ is a quadric threefold over $k(t)$. 
    One might wonder if there exist a number field $k$, a prime number $\ell$, a non-trivial mod $\ell$ symbol $s$ of degree $n+1$, and a norm variety $X$ for $s$ for which $\cl\colon CH^2(X)[\ell]\to H^4(X,\Z_\ell(2))$ is not injective. If $\ell$ is odd, this is impossible, as $\cd_{\ell}(k)=2$. Suppose now that $\ell=2$, so that $X$ is the quadric hypersurface associated to a Pfister neighbor $q$ of rank $2^n+1$. By \cite[Theorem 6.1]{karpenko1990algebro}, $CH^2(X)_{\on{tors}}$ is either $0$ or $\Z/2$. Let $\mc{R}$ be the Rost motive of $X$: it is a direct summand of $M(X;\Z_{(2)})$. By \cite[Theorem RM.10]{karpenko2013standard}, $CH^2(\mc{R})[2]\neq 0$ if and only if there exists $1\leq i\leq n-1$ such that $2^n-2^i=2$, that is, if and only if $n=2$. If this is the case, then $CH^2(X)_{\on{tors}}\simeq \Z/2$ and $\dim(X)=3$.
    
    By definition of norm variety, for every field extension $F/k$ we have $X(F)\neq\emptyset$ (that is, $q_F$ is isotropic) if and only if $s_F$ is trivial. Therefore, by a theorem of Rost \cite[Theorem 5]{rost1990some} (see also \cite[Lemma 2.1]{yagita2008chow}, or follow the construction of the isomorphisms in \cite[Theorem RM.10]{karpenko2013standard}), the natural pull-back map $CH^*(\mc{R})\to CH^*(\mc{R}_F)$ is injective for all field extensions $F/k$ such that $q_F$ is anisotropic. 
    
    Recall that every form of degree $5$ over a $p$-adic field is isotropic; see \cite[XI, Example 6.2(4)]{lam2005introduction}. Thus, if $q$ is isotropic at all real places of $k$, then $q$ is isotropic at all places of $k$, and so it is isotropic by the Hasse-Minkowski principle \cite[VI, Principle 3.1]{lam2007exercises}, hence $CH^2(X)$ is torsion-free by \cite[Theorem 6.1]{karpenko1990algebro}. Suppose now that there exists one real embedding $k\subset \mathbb{R}$ such that $q_{\R}$ is not isotropic. We have a commutative square
    \[
    \begin{tikzcd}
    CH^2(X)/2 \arrow[r,"\sim"] \arrow[d,"\cl"]  & CH^2(X_{\R})/2\arrow[d, hook,"\cl"] \\
    H^4(X,\Z/2) \arrow[r] & H^4(X_{\R},\Z/2) 
    \end{tikzcd}
    \]
    where the vertical maps are the cycle class maps in \'etale cohomology and the horizontal maps are induced by base change. The vertical map on the right is injective by \cite[Proposition 2.5]{colliot1995unramified}. Since $CH^2(X)_{\tors}\simeq \Z/2$, we deduce that $\cl\colon CH^2(X)_{\tors}\to H^4(X,\Z/2)$ is injective, and so 
    %the composition $CH^2(X)_{\tors}\to H^4(X,\Z_2(2))\to H^4(X,\Z/2)$
    $\cl\colon CH^2(X)_{\tors}\rightarrow H^4(X,\Z_2(2))$
    is also injective.
\end{rem}

\begin{comment}
\begin{rem}
    (1) Let $X/k$ be as in the proof of \Cref{quadric2}. The group $T_{\ell}(\on{Br}(k))$ is a free $\Z_{\ell}$-module of countably infinite rank. Therefore (\ref{f2f3}) implies that $H^4(X,\Z_\ell(2))$ is a free $\Z_\ell$-module of countably infinite rank. On the other hand $\on{CH}^2(X)_{\on{tors}}\simeq \Z/2$. Nevertheless, the cycle class map $\on{CH}^2(X)_{\on{tors}}\to H^4(X,\Z_2(2))$ is the zero map.

    (2) In \Cref{quadric2} it is not necessary to pass to finite field extension of the base field.
    %(3) \Cref{quadric2} does not contradict ??? (Merkurjev-Suslin), because the class $\alpha$ in the kernel of the cycle class map vanishes over $\ov{k}$.
\end{rem}
\end{comment}
	
	\section*{Acknowledgements}

We thank Burt Totaro for telling us about \Cref{schreieder}. We thank Jean-Louis Colliot-Th\'el\`ene, Stefan Schreieder, 
%and 
Burt Totaro,
and the referee
for helpful comments and suggestions on this work.

\end{document}